\documentclass{amsart}
\usepackage[english]{babel}
\usepackage{amsmath, amsthm, amssymb}
\usepackage{enumerate}
\usepackage{geometry}
\usepackage{hyperref}
\usepackage{cleveref}

\theoremstyle{plain}
\newtheorem{theorem}{Theorem}[section]
\newtheorem{proposition}[theorem]{Proposition}
\newtheorem{lemma}[theorem]{Lemma}

\theoremstyle{definition}
\newtheorem{definition}[theorem]{Definition}

\newcommand{\C}{\mathbb{C}}

\newcommand{\Q}{\mathbb{Q}}
\newcommand{\Z}{\mathbb{Z}}

\newcommand{\F}{\mathbb{F}}

\let\H\relax
\newcommand{\H}{\mathcal{H}}

\DeclareMathOperator{\tr}{tr}

\DeclareMathOperator{\SL}{SL}

\DeclareMathOperator{\Aut}{Aut}
\DeclareMathOperator{\lcm}{lcm}
\newcommand{\legendre}[2]{\ensuremath{\left( \frac{#1}{#2} \right) }}

\title{Sums of Hurwitz class numbers, CM modular forms, and primes of the form $x^2+ny^2$}
\author{Mikul\'{a}\v{s} Zindulka}
\date{\today}
\address{Charles University, Faculty of Mathematics and Physics, Department of Algebra,
Sokolovsk\'{a} 83, 186 75 Praha 8, Czech Republic}
\email{mikulas.zindulka@matfyz.cuni.cz}
\thanks{The author acknowledges support by Czech Science Foundation (GA\v{C}R) grant 21-00420M, Charles University project PRIMUS/24/SCI/010, GAUK project No. 134824, and the Charles University Research Centre program UNCE/24/SCI/022.}

\begin{document}

\begin{abstract}
We consider sums of Hurwitz class numbers of the type $ \sum_{t \equiv m \pmod{M}}H(4p-t^2) $, where $ M $ is composite. For $ M = 6 $ and $ 8 $, we show that these sums can be expressed in terms of coefficients of CM cusp forms. This leads to explicit formulas which depend on the expression of $ p $ in the form $ x^2+ny^2 $.
\end{abstract}

\maketitle

\section{Introduction}
\label{secIntro}

Modular forms found many applications in and outside of number theory. Given an arithmetic function $ a(n) $ for $ n \in \Z_{\geq 0} $, let $ f(\tau) := \sum_{n = 0}^\infty a(n)q^n $ be its generating function, where $ q = e^{2\pi i \tau} $. Provided that the series converges for $ \tau $ in the upper half-plane $ \mathbb{H} $, we may hope that $ f $ is a modular form, or that it at least satisfies a weight $ \kappa $ modularity for some $ \kappa \in \frac{1}{2}\Z $.

The example relevant for this paper is the generating function
\[
	\H(\tau) := \sum_{n=0}^\infty H(n)q^n,
\]
where $ H(n) $ is the $ n $th \emph{Hurwitz class number} (defined in \Cref{secPrelim}). It turns out that $ \H $ is not a modular form but a weight $ \frac{3}{2} $ mock modular form. Mock modular forms have been the subject of an ongoing research interest, see for example \cite{BO1, BO2, BF, Zw}.

One motivation for studying Hurwitz class numbers comes from their connection to elliptic curves. If $ p $ is a prime, $ r \in \Z_{\geq 1} $, and $ E $ is an elliptic curve over the finite field $ \F_{p^r} $, then its \emph{trace of Frobenius} is defined by
\[
	\tr(E) := p^r+1-\#E\left(\F_{p^r}\right).
\]
By the Hasse bound, $ |\tr(E)| \leq 2\sqrt{p^r} $. When counting elliptic curves, it is quite natural to weight $ E $ by $ 1 $ over the size of its automorphism group. Thus, for $ t \in \Z $ we let
\[
	N_A(p^r; t) := \sum_{\substack{E/\F_{p^r}\\\tr(E) = t}}\frac{1}{\#\Aut_{\F_{p^r}}(E)},
\]
where the sum is taken over all elliptic curves over $ \F_{p^r} $ with trace of Frobenius equal to $ t $. If $ p \geq 3 $, $ |t| \leq 2\sqrt{p^r} $ and $ p \nmid t $, then by \cite[Theorem 3]{KP18},
\[
	2 N_A(p^r; t) = H(4p^r-t^2).
\]
The last equality can be used to show the identity
\[
	\sum_{t \in \Z}H(4p-t^2) = 2p
\]
for a prime $ p $ \cite[p. 154]{Ei}.

More generally, one can consider the sums
\[
	H_{m,M}(n) := \sum_{\substack{t \in \Z\\t \equiv m \pmod{M}}}H(4n-t^2),
\]
where $ m \in \Z $ and $ M \in \Z_{\geq 1} $. In \cite{Br}, Brown et al. proved explicit formulas for $ H_{m,M}(p) $ when $ p $ is a prime and $ M \in \{2,3,4\} $. They also conjectured a formula for $ M = 5 $ and gave an incomplete formula for $ M = 7 $. For example, if $ M = 5 $ and $ m = 0 $, then
\[
	H_{0,5}(p) = \begin{cases}
		\frac{p+1}{2}&\text{if }p \equiv 1 \pmod{5},\\
		\frac{p+1}{3}&\text{if }p \equiv 2, 3 \pmod{5},\\
		\frac{p-3}{2}&\text{if }p \equiv 4 \pmod{5}.
	\end{cases}
\]
Bringmann and Kane \cite{BK} developed a unified approach to these sums for $ M $ prime, proved the conjecture for $ M = 5 $ and a complete formula for $ M = 7 $. Their technique was further simplified by Mertens \cite{Me2} as an application of his development of the theory of Rankin-Cohen brackets of weight $ \frac{3}{2} $ mock modular forms and weight $ \frac{1}{2} $ modular forms.

Besides sums of Hurwitz class numbers, one can also consider higher moments. For $ \kappa \in \Z_{\geq 0} $, $ m \in \Z $, $ M \in \Z_{\geq 1} $, and $ n \in \Z_{\geq 0} $, let
\[
	H_{\kappa,m,M}(n) := \sum_{\substack{t \in \Z\\t \equiv m \pmod{M}}}H(4n-t^2)t^\kappa.
\]

Kane and Pujahari gave an asymptotic formula for the even moments $ H_{2k,m,M}(n) $ expressed in terms of the zeroth moment \cite[Theorem 1.1]{KP22}, which they subsequently applied to elliptic curves, namely the weighted even moments with respect to the trace of Frobenius \cite[Theorem 1.2]{KP22}. Odd moments were similarly treated in \cite[Theorem 1.4]{BKP}.

In this paper, we find explicit formulas for $ H_{m,M}(p) $ for the composite moduli $ M = 6 $ and $ M = 8 $. Let us state our results for $ m = 0 $. The following is a famous fact from elementary number theory \cite[p. 7]{Co}: If $ p $ is a prime, then
\begin{alignat*}{4}
	p = x^2+y^2,\quad x, y \in \Z &\qquad\Longleftrightarrow\qquad&&p = 2\text{ or }p \equiv 1 \pmod{4},\\
	p = x^2+2y^2,\quad x, y \in \Z&\qquad\Longleftrightarrow\qquad&&p = 2\text{ or }p \equiv 1, 3 \pmod{8},\\
	p = x^2+3y^2,\quad x, y \in \Z&\qquad\Longleftrightarrow\qquad&&p = 3\text{ or }p \equiv 1 \pmod{3}.
\end{alignat*}
If $ p $ is an odd prime such that $ p = x_0^2+y_0^2 $ with $ x_0, y_0 \in \Z $, then exactly one of $ x_0, y_0 $ is odd, hence $ p $ is also of the form $ x^2+4y^2 $. If we require $ x, y \in \Z_{\geq 0} $, then this expression is unique. More generally, if $ n \in \Z_{\geq 2} $ and $ p $ is an odd prime of the form $ x^2+ny^2 $, then there exists a unique solution to the equation $ p = x^2+ny^2 $ with $ x, y \in \Z_{\geq 0} $ \cite[Exercise 3.23]{Co}.

\begin{theorem}
\label{thmH06}
Let $ p \geq 5 $ be a prime. For $ p \equiv 1 \pmod{3} $, let $ p = x^2+3y^2 $, where $ x, y \in \Z $. If $ \chi_{-3} $ denotes the non-principal character modulo $ 3 $, then
\[
	H_{0,6}(p) = \begin{cases}
		\frac{p+1}{3}+\frac{1}{3}\chi_{-3}(x)x&\text{if }p \equiv 1 \pmod{3},\\
		\frac{2p-4}{3}&\text{if }p \equiv 2 \pmod{3}.
	\end{cases}
\]
\end{theorem}
\begin{proof}
This is a special case of \Cref{thmH6}.
\end{proof}

\begin{theorem}
\label{thmH08}
Let $ p \geq 3 $ be a prime. For $ p \equiv 1 \pmod{4} $, let $ p = x^2+4y^2 $, where $ x, y \in \Z $. If $ \chi_{-4} $ denotes the non-principal character modulo $ 4 $, then
\[
	H_{0,8}(p) = \begin{cases}
		\frac{p+1}{4}+\frac{1}{2}\chi_{-4}(x)x&\text{if }p \equiv 1 \pmod{4},\\
		\frac{p+1}{3}&\text{if }p \equiv 3 \pmod{8},\\
		\frac{p-3}{2}&\text{if }p \equiv 7 \pmod{8}.
	\end{cases}
\]
\end{theorem}
\begin{proof}
This is a special case of \Cref{thmH81}.
\end{proof}
We prove formulas of this type for all other $ m $. It is interesting to note that for $ M \leq 5 $, the formulas are linear in $ p $, while in our theorems we have the ``error terms" $ \frac{1}{3}\chi_{-3}(x)x $ and $ \frac{1}{2}\chi_{-4}(x)x $ of size $ O(\sqrt{p}) $. As will become clear from the proof, they appear as coefficients of CM cusp forms. For example, $ \frac{1}{3}\chi_{-3}(x)x $ comes from $ \frac{1}{6}\Psi_3(\chi_{-3},\tau) $, where
\[
	\Psi_3(\chi_{-3},\tau) := \frac{1}{2}\sum_{n=1}^\infty\left(\sum_{x^2+3y^2 = n}\chi_{-3}(x)x\right)q^n = q-4q^7+2q^{13}+8q^{19}+O\left(q^{20}\right)
\]
(the inner sum runs over all $ x, y \in \Z $ such that $ x^2+3y^2 = n $). This is a newform in the space $ S_2(\Gamma_0(36)) $ with complex multiplication by $ \Q(\sqrt{-3}) $, as we show in \Cref{lemmaCM}. A similar situation occurs for $ M = 7 $: the sum $ H_{0,7}(p) $ depends on the coefficient of a newform in $ S_2(\Gamma_0(49)) $ with complex multiplication by $ \Q(\sqrt{-7}) $ \cite[Corollary 4.5]{BK}.

The rest of the paper is organized as follows. In \Cref{secPrelim}, we review the preliminaries about Hurwitz class numbers, modular forms, and CM forms. \Cref{secAux} contains auxiliary lemmas. The explicit formulas for $ H_{m,M}(p) $ are proved in \Cref{secExp}.

\section*{Acknowledgments}
This work was accomplished during a research visit to Professor Ben Kane at the University of Hong Kong. I would like to thank the Department of Mathematics for their hospitality and Ben Kane for his invitation and countless helpful suggestions.

\section{Preliminaries}
\label{secPrelim}

We begin by defining Hurwitz class numbers. Let $ \mathcal{Q}_D $ be the set of binary integral quadratic forms of discriminant $ D $. For $ Q \in \mathcal{Q}_D $, let $ \Gamma_Q $ be the stabilizer group of $ Q $ in $ \SL_2(\Z) $.

\begin{definition}
\label{defH}
Let $ D \in \Z_{<0} $, $ D \equiv 0,1 \pmod{4} $. The $ |D| $-th \emph{Hurwitz class number} is defined by
\[
	H(|D|) = \sum_{Q \in \mathcal{Q}_D/\SL_2(\Z)}\frac{1}{\omega_Q},
\]
where
\[
	\omega_Q := \frac{|\Gamma_Q|}{2} = \begin{cases}
		2&\text{if }Q = a(x^2+y^2),\\
		3&\text{if }Q = a(x^2+xy+y^2),\\
		1&\text{otherwise}.
	\end{cases}
\]
We also set $ H(0) := -\frac{1}{12} $ and $ H(n) := 0 $ if $ n \in \Z_{<0} $ or $ n \equiv 1, 2 \pmod{4} $.
\end{definition}

We briefly summarize some basic facts about modular forms; for details, see \cite{Ko,On}. Let $ F: \mathbb{H} \to \C $ denote a function on the upper half-plane. The \emph{weight $ \kappa $ slash operator} is defined as follows: if $ \kappa \in \Z $, then
\[
	\left(F\vert_\kappa\gamma\right)(\tau) := (c\tau+d)^{-\kappa}F(\gamma\tau),
\]
where $ \gamma = \begin{pmatrix}a&b\\c&d\end{pmatrix} \in \SL_2(\Z) $, and if $ \kappa \in \frac{1}{2}+\Z $, then
\[
	\left(F\vert_\kappa \gamma\right)(\tau) := \legendre{c}{d}^{2\kappa}\varepsilon_d^{2\kappa}(c\tau+d)^{-\kappa}F(\gamma\tau),
\]
where $ \gamma = \begin{pmatrix}a&b\\c&d\end{pmatrix} \in \Gamma_0(4) $, $ \legendre{\cdot}{\cdot} $ is the \emph{extended Legendre symbol}, and
\[
	\varepsilon_d = \begin{cases}
		1&\text{if }d \equiv 1 \pmod{4},\\
		i&\text{if }d \equiv 3 \pmod{4}.
	\end{cases}
\]

Let $ \Gamma $ be a congruence subgroup containing $ T := \begin{pmatrix}1&1\\0&1\end{pmatrix} $ and such that $ \Gamma \subset \Gamma_0(4) $ if $ \kappa \in \frac{1}{2}+\Z $. We say that $ F $ satisfies \emph{modularity of weight $ \kappa $} on $ \Gamma $ with Nebentypus character $ \chi $ if
\[
	F\vert_\kappa = \chi(d)F.
\]

Moreover, if $ F $ is holomorphic on $ \mathbb{H} $ and $ F(\tau) $ grows at most polynomially in $ v $ as $ \tau := u+iv \to \Q \cup \{i\infty\} $, then $ F $ is called a \emph{(holomorphic) modular form}. If $ F $ vanishes at the cusps of $ \Gamma $, then it is called a \emph{cusp form}. The modular (resp. cusp) forms of weight $ \kappa $ on $ \Gamma $ with character $ \chi $ form a $ \C $-vector space which will be denoted by $ M_\kappa(\Gamma, \chi) $ (resp. $ S_\kappa(\Gamma, \chi) $).

If $ F $ satisfies weight $ \kappa $ modularity and there exist holomorphic functions $ F_j $ for $ 0 \leq j \leq \ell $ such that $ F(\tau) = \sum_{j=0}^\ell F_j(\tau)v^{-j} $, then $ F $ is called an \emph{almost holomorphic modular form} and the function $ F_0 $ is called a \emph{quasimodular form}.

The \emph{$ U $-operator} is defined for $ M \in \Z_{\geq 1} $ and $ F(\tau) = \sum_{n=0}^\infty a(n)q^n $ by
\[
	\left(F\vert U_M\right)(\tau) := \sum_{n=0}^\infty a(Mn)q^n,
\]
and the \emph{$ V $-operator} is defined for $ M \in \Z_{\geq 1} $ and $ F: \mathbb{H}\to \C $ by
\[
	\left(F\vert V_M\right)(\tau) := F(M\tau).
\]
Suppose that $ F $ satisfies weight $ \kappa $ modularity on $ \Gamma_0(N_1) \cap \Gamma_1(N_2) $ with character $ \chi $, where $ N_2 \mid N_1 $ and where we assume $ 4 \mid N_1 $ if $ \kappa \in \frac{1}{2}+\Z $. Then $ F\vert U_M $ satisfies weight $ \kappa $ modularity on $ \Gamma_0\left(\lcm(N_1,M)\right)\cap \Gamma_1(N_2) $ with character $ \chi \cdot \legendre{M}{\cdot}^{2\kappa} $, and $ F\vert V_M $ satisfies weight $ \kappa $ modularity on $ \Gamma_0(N_1M)\cap \Gamma_1(N_2) $ with character $ \chi \cdot \legendre{M}{\cdot}^{2\kappa} $.

For a Dirichlet character $ \psi $ with modulus $ M $, the \emph{$ \psi $-twist} of $ F = \sum_{n=0}^\infty a(n)q^n $ is defined by
\[
	\left(F \otimes \psi\right)(\tau) := \sum_{n=0}^\infty \psi(n)a(n)q^n.
\]
If $ F \in M_\kappa(\Gamma_0(N),\chi) $, then $ F \otimes \psi \in M_\kappa(\Gamma_0(NM^2),\chi\psi^2) $.

For $ M \in \Z_{\geq 1} $ and $ m \in \Z $, the \emph{sieving operator $ S_{M,m} $} is defined by
\[
	\left(F\vert S_{M,m}\right)(\tau) = \sum_{n \in \Z}a(Mn+m)q^{Mn+m}.
\]
We will show how this operator acts on spaces of modular forms in \Cref{lemmaS}.

To check whether two modular forms in the same space are equal, we use the Sturm bound, which is a consequence of the valence formula.

\begin{theorem}[Sturm bound, {\cite[Corollary 9.20]{St}}]
\label{thmSturm}
Let $ \kappa \in \Z_{\geq 2} $, $ \Gamma \subset \SL_2(\Z) $ be a congruence subgroup, and $ \chi $ be a character. If $ f = \sum_{n=0}^\infty a(n)q^n $ and $ g = \sum_{n=0}^\infty b(n)q^n $ are two modular forms in $ M_\kappa(\Gamma,\chi) $ such that
\[
	a(n) = b(n)\text{ for }0 \leq n \leq \left\lfloor\frac{\kappa m}{12}\right\rfloor,
\]
where $ m = [\SL_2(\Z):\Gamma] $, then $ f = g $.
\end{theorem}

In the situations where we apply \Cref{thmSturm}, we use the following simple lemma to compute the index of the group $ \Gamma $.

\begin{lemma}[{\cite[Lemma 2.2]{BK2}}]
\label{lemmaInd}
If $ N_1, N_2 \in \Z_{\geq 1} $ with $ N_2 \mid N_1 $, then
\[
	\left[\SL_2(\Z):\Gamma_0(N_1)\cap\Gamma_1(N_2)\right] = N_1\prod_{p\mid N_1}\left(1+\frac{1}{p}\right)\varphi(N_2),
\]
where the product runs over all primes dividing $ N_1 $. In particular,
\[
	\left[\SL_2(\Z):\Gamma_0(N_1)\right] = N_1\prod_{p \mid N_1}\left(1+\frac{1}{p}\right).
\]
\end{lemma}

Next, we recall the definition of CM modular form.

\begin{definition}[{\cite[p. 34]{Ri}}]
Let $ \kappa \in \Z_{\geq 2} $, $ N \in \Z_{\geq 1} $, and $ \chi $ be a Dirichlet character. Let $ f(\tau) = \sum_{n=0}^\infty c(n)q^n $ in $ S_\kappa(\Gamma_0(N),\chi) $, and assume that $ f $ is a newform. Let $ \varphi $ be a Dirichlet character modulo $ M $ for some $ M \in \Z_{\geq 2} $. The form $ f $ has \emph{complex multiplication by $ \varphi $} if
\[
	\varphi(p)a(p) = a(p)
\]
for all primes $ p $ in a set of primes of density $ 1 $.

If $ D \in \Z_{<0} $, $ D \equiv 0, 1 \pmod{4} $, $ K = \Q(\sqrt{D}) $ is a quadratic field, and $ \varphi = \legendre{D}{\cdot} $ is the quadratic character of $ K $, then we say that $ f $ \emph{has complex multiplication by $ K $}.
\end{definition}

The following theorem due to Hirzebruch and Zagier states that one can complete the generating function $ \H $ for the Hurwitz class numbers to obtain a \emph{harmonic weak Maass form}. Since we will not use this theorem directly, we do not define harmonic weak Maass forms and instead refer to, e.g., \cite[Definition 3.1]{Me}.

\begin{theorem}[{\cite[Theorem 2]{HZ}}]
\label{thmHhat}
The function
\[
	\widehat{\H}(\tau) := \H(\tau)+\frac{1}{8\pi\sqrt{v}}+\frac{1}{4\sqrt{\pi}}\sum_{n=1}^\infty n\Gamma\left(-\frac{1}{2},4\pi n^2 v\right)q^{-n^2},
\]
where $ \Gamma(\cdot, \cdot) $ denotes the \emph{incomplete Gamma function}, is a harmonic Maass form of weight $ \frac{3}{2} $ on $ \Gamma_0(4) $.
\end{theorem}

If we let
\[
	\theta_{m,M}(\tau) := \sum_{\substack{n \in \Z\\n\equiv m \pmod{M}}}q^{n^2},
\]
then the generating function
\[
	\H_{m,M}(\tau) := \sum_{n=0}^\infty H_{m,M}(n)q^n
\]
satisfies
\[
	\H_{m,M} = \left(\H\theta_{m,M}\right)\vert U_4.
\]

\begin{definition}[{\cite[p. 5]{KP22}}]
\label{defRC}
For $ F_1 $, $ F_2 $ transforming like modular forms of weight $ \kappa_1, \kappa_2 \in \frac{1}{2}\Z $, respectively, define for $ k \in \Z_{\geq 0} $ the $ k $-th \emph{Rankin-Cohen bracket}
\[
	[F_1, F_2]_k := \frac{1}{(2\pi i)^k}\sum_{j=0}^k (-1)^j\binom{\kappa_1+k-1}{k-j}\binom{\kappa_2+k-1}{j}F_1^{(j)}F_2^{(k-j)}
\]
with $ \binom{\alpha}{j} := \frac{\Gamma(\alpha+1)}{j! \Gamma(\alpha-j+1)} $. Then $ [F_1, F_2]_k $ transforms like a modular form of weight $ \kappa_1+\kappa_2+2k $.
\end{definition}
Note that  for $ k = 0 $, the Rankin-Cohen bracket $ [F_1,F_2]_0 $ is simply the product $ F_1F_2 $. Let
\[
	\Lambda_{\ell,m,M}(\tau) := \sum_{n=1}^\infty \lambda_{\ell,m,M}(n)q^n,\qquad\text{where}\qquad\lambda_{\ell,m,M}(n) := \sum_{\pm}\sideset{}{^*}\sum_{\substack{t>s\geq 0\\t^2-s^2=n\\t \equiv \pm m\pmod{M}}}(t-s)^\ell.
\]
The $^*$ next to the sum means that terms for $ s = 0 $ are taken with weight $ \frac{1}{2} $.

\begin{lemma}[{\cite[Lemma 3.3]{KP22}}]
\label{lemmaHLambda}
For $ k \in \Z_{\geq 0} $, $ m \in \Z $, and $ M \in \Z_{\geq 1} $, the function
\[
	\left([\H,\theta_{m,M}]_k+2^{-1-2k}\binom{2k}{k}\Lambda_{2k+1,m,M}\right)\bigg\vert U_4
\]
is a holomorphic cusp form of weight $ 2+2k $ on $ \Gamma_0(4M^2)\cap \Gamma_1(M) $ (resp. $ \Gamma_0(4M^2) $) if $ M \nmid m $ (resp. $ M \mid m $) if $ k > 0 $ and quasimodular on that group if $ k = 0 $.
\end{lemma}
This result was originally stated by Mertens \cite[p.~379]{Me2} in the special case when $ M = p $ is prime. Kane and Pujahari \cite{KP22} pointed out that it remains valid for $ M $ composite and corrected the constant appearing in front of $ \Lambda_{2k+1,m,M} $.

\section{Auxiliary lemmas}
\label{secAux}

\Cref{lemmaHLambda} for $ k = 0 $ is the main ingredient in our proof of explicit formulas like the one in \Cref{thmH06}. In this section, we find a different expression for $ \Lambda_{\ell,m,M} $ which is more suitable for calculations. Similar formulas were given by Mertens in the case when $ M = p $ is a prime, see \cite[Proposition 7.2]{Me2} and also \cite[Proposition V.4.3]{Me}. It would be sufficient to evaluate $ \Lambda_{\ell,m,M} $ only for $ \ell = 1 $. However, finding the evaluation for general $ \ell $ presents very little additional difficulty.

For $ \ell \in \Z_{\geq 0} $, $ a, b \in \Z $, and $ M, n \in \Z_{\geq 1} $, let
\[
	\mu_{\ell, a, b, M}(n) := \sum_{\substack{t>s\geq 1\\t^2-s^2=4n\\t \equiv a \pmod{M}\\s \equiv b \pmod{M}}}(t-s)^\ell.
\] 

\begin{lemma}
\label{lemmaLambda1}
Let $ \ell \in \Z_{\geq 0} $, $ a, b \in \Z $, and $ M = 2^eM_1 $, where $ e \in \Z_{\geq 1} $ and $ M_1 \in \Z_{\geq 1} $ is odd. Let $ f = 0 $ if $ e = 1 $ and $ f = e $ if $ e \geq 2 $. Let $ n \in \Z_{\geq 1} $ be such that $ \gcd(n,M) = 1 $.
\begin{enumerate}
\item If $ a $ or $ b $ is odd, then $ \mu_{\ell, a, b, M}(n) = 0 $.
\item If $ a = 2a_1 $ and $ b = 2b_1 $, then
\[
	\mu_{\ell,a,b,M}(n) = 2^\ell\sum_{\substack{d\mid n\\d<\sqrt{n}\\d \equiv a_1-b_1\pmod{2^{e-1}M_1}}}d^\ell
\]
if $ n \equiv a_1^2-b_1^2 \pmod{2^fM_1} $ and $ \mu_{\ell,a,b,M}(n) = 0 $ otherwise.
\end{enumerate}
\end{lemma}
\begin{proof}
To prove (1), note that since we assume $ \gcd(n,M) = 1 $, the equality $ \frac{(t-s)}{2}\frac{(t+s)}{2} = n $ implies that $ \frac{t-s}{2} $ and $ \frac{t+s}{2} $ are odd, hence $ t $ and $ s $ are even.

Next, we prove (2). First, suppose $ e = 1 $, hence $ M = 2M_1 $ for $ M_1 $ odd. If $ \mu_{\ell,a,b,M}(n) $ is non-zero, then there exist $ t>s\geq 1 $, $ t \equiv 2a_1 \pmod{M} $, $ s \equiv 2b_1 \pmod{M} $ such that $ t^2-s^2 = 4n $, hence
\[
	n = \frac{(t-s)}{2}\frac{(t+s)}{2} \equiv (a_1-b_1)(a_1+b_1) \pmod{M_1}.
\]

We show that if $ n \equiv a_1^2-b_1^2 \pmod{M_1} $, then
\begin{equation}
\label{eqSumts}
	\sum_{\substack{t>s\geq 1\\t^2-s^2 = 4n\\t\equiv a \pmod{M}\\s \equiv b \pmod{M}}}(t-s)^\ell = 2^\ell\sum_{\substack{t>s\geq 1\\\frac{(t-s)}{2}\frac{(t+s)}{2} = n\\t \equiv 2a_1 \pmod{M}\\s \equiv 2b_1 \pmod{M}}}\left(\frac{t-s}{2}\right)^\ell = 2^\ell\sum_{\substack{d \mid n\\d<\sqrt{n}\\d \equiv a_1-b_1\pmod{M_1}}}d^\ell.
\end{equation}
Let us prove the second equality. If $ t > s \geq 1 $ are such that $ t \equiv 2a_1 \pmod{M} $, $ s \equiv 2b_1 \pmod{M} $, and $ \frac{(t-s)}{2}\frac{(t+s)}{2} = n $, then we set $ d = \frac{t-s}{2} $. We have $ d \mid n $, $ d < \sqrt{n} $, and $ d \equiv a_1-b_1 \pmod{M_1} $. Conversely, let $ d \mid n $, $ d < \sqrt{n} $ be such that $ d \equiv a_1-b_1 \pmod{M_1} $. By assumption, $ n $ is coprime to $ M $, hence $ d $ is invertible modulo $ M_1 $ and $ \frac{n}{d} \equiv a_1+b_1 \pmod{M_1} $. We find $ t $ and $ s $ such that $ \frac{t-s}{2} = d $ and $ \frac{t+s}{2} = \frac{n}{d} $, which satisfy $ t = \frac{n}{d}+d \equiv 2a_1 \pmod{M_1} $ and $ s = \frac{n}{d}-d \equiv 2b_1 \pmod{M_1} $. Moreover, since $ n $ is odd, $ d $ and $ \frac{n}{d} $ are odd, while $ t $ and $ s $ are even. Thus, $ t \equiv 2a_1 \pmod{2M_1} $ and $ s \equiv 2b_1 \pmod{2M_1} $.

Secondly, suppose $ e \geq 2 $. If $ \mu_{\ell,a,b,M}(n) $ is non-zero, then there exist $ t > s \geq 1 $, $ t \equiv 2a_1 \pmod{M} $, $ s \equiv 2b_1 \pmod{M} $ such that $ t^2-s^2 = 4n $. If we let $ t = Mv+2a_1 $ and $ s = Mu+2b_1 $, then
\begin{align*}
	 n& = \frac{(t-s)}{2}\frac{(t+s)}{2} = \left(2^{e-1}M_1(v-u)+a_1-b_1\right)\left(2^{e-1}M_1(v+u)+a_1+b_1\right)\\
	 & \equiv 2^{e-1}M_1(v-u)(a_1+b_1)+2^{e-1}M_1(v+u)(a_1-b_1)+(a_1-b_1)(a_1+b_1) \equiv a_1^2-b_1^2 \pmod{M}.
\end{align*}

We show that if $ n \equiv a_1^2-b_1^2 \pmod{M} $, then
\[
	\sum_{\substack{t>s\geq 1\\t^2-s^2 = 4n\\t\equiv a \pmod{M}\\s \equiv b \pmod{M}}}(t-s)^\ell = 2^\ell\sum_{\substack{t>s\geq 1\\\frac{(t-s)}{2}\frac{(t+s)}{2} = n\\t \equiv 2a_1 \pmod{M}\\s \equiv 2b_1 \pmod{M}}}\left(\frac{t-s}{2}\right)^\ell = 2^\ell\sum_{\substack{d \mid n\\d<\sqrt{n}\\d \equiv a_1-b_1\pmod{2^{e-1}M_1}}}d^\ell.
\]
Again, let us prove the second equality. If $ t > s \geq 1 $ are such that $ t \equiv 2a_1 \pmod{M} $, $ s \equiv 2b_1 \pmod{M} $, and $ \frac{(t-s)}{2}\frac{(t+s)}{2} = n $, then we set $ d = \frac{t-s}{2} $. We have $ d \mid n $, $ d < \sqrt{n} $, and $ d \equiv a_1-b_1 \pmod{2^{e-1}M_1} $. Conversely, let $ d \mid n $, $ d < \sqrt{n} $ be such that $ d \equiv a_1-b_1 \pmod{2^{e-1}M_1} $. We find $ t $ and $ s $ such that $ \frac{t-s}{2} = d $ and $ \frac{t+s}{2} = \frac{n}{d} $. Either $ d \equiv a_1-b_1 \pmod{M} $ or $ d \equiv 2^{e-1}M_1+a_1-b_1 \pmod{M} $. In the first case, $ \frac{n}{d} \equiv a_1+b_1 \pmod{M} $, hence $ t = \frac{n}{d}+d \equiv 2a_1 \pmod{M} $ and $ s = \frac{n}{d}-d \equiv 2b_1 \pmod{M} $. In the second case,
\[
	\left(2^{e-1}M_1+a_1-b_1\right)\left(2^{e-1}M_1+a_1+b_1\right) \equiv 2^{e-1}{M_1}\cdot 2a_1+(a_1-b)(a_1+b_1) \equiv n \pmod{M},
\]
hence $ \frac{n}{d} \equiv 2^{e-1}M+a_1+b_1 $, and again $ t \equiv 2a_1 \pmod{M} $, $ s \equiv 2b_1 \pmod{M} $.
\end{proof}

Let
\[
	G_{\ell,m,M}(\tau) := \sum_{n=1}^\infty g_{\ell,m,M}(n)q^n,\qquad\text{where}\qquad g_{\ell,m,M}(n) := \sum_{\pm}\sum_{\substack{d\mid n\\d<\sqrt{n}\\d \equiv \pm m \pmod{M}}}d^\ell
\]
and
\[
	T_{\ell,m,M}(\tau) := \sum_{\pm}\sum_{\substack{n=1\\n \equiv \pm m \pmod{M}}}^\infty n^\ell q^{n^2}.
\]

\begin{lemma}
\label{lemmaLambda2}
Let $ \ell \in \Z_{\geq 0} $ and $ M = 2^eM_1 $, where $ e \in \Z_{\geq 1} $ and $ M_1 \in \Z_{\geq 1} $ is odd. Let $ f = 0 $ if $ e = 1 $ and $ f = e $ if $ e \geq 2 $ and let $ \chi_{M,0} $ denote the principal character modulo $ M $.
\begin{enumerate}
\item If $ m \in \Z $ is odd, then
\[
	\Lambda_{\ell,m,M}\vert U_4\otimes\chi_{M,0} = 0.
\]
\item If $ m = 2m_1 $, where $ m_1 \in \Z $, then
\[
	\Lambda_{\ell,m,M}\vert U_4\otimes\chi_{M,0} = 2^\ell\sum_{\substack{b_1 \pmod{2^{e-1}M_1}\\(m_1^2-b_1^2,2^{e-1}M_1)=1}}G_{\ell,m_1-b_1,2^{e-1}M_1}\vert S_{2^fM_1,m_1^2-b_1^2}\vert S_{2,1}+2^{\ell-1}T_{\ell,m_1,2^{e-1}M_1}\otimes \chi_{M,0}.
\]
\end{enumerate}
\end{lemma}
\begin{proof}
By definition, the left-hand side equals
\[
	\sum_{\substack{n=1\\(n,M)=1}}^\infty \lambda_{\ell,m,M}(4n)q^n = \sum_{\substack{n=1\\(n,M)=1}}^\infty \sum_{\pm}\left(\sum_{\substack{t>s\geq 1\\t^2-s^2=4n\\t \equiv \pm m \pmod{M}}}(t-s)^\ell+\frac{1}{2}\sum_{\substack{t\geq 1\\t^2 = 4n\\t \equiv \pm m \pmod{M}}}t^\ell\right)q^n.
\]
Let us split this sum into two, namely
\begin{align*}
	\Sigma_1& := \sum_{\substack{n=1\\(n,M)=1}}^\infty \sum_{\pm}\left(\sum_{\substack{t>s\geq 1\\t^2-s^2=4n\\t \equiv \pm m \pmod{M}}}(t-s)^\ell\right)q^n,\\
	\Sigma_2& := \sum_{\substack{n=1\\(n,M)=1}}^\infty \sum_{\pm}\left(\frac{1}{2}\sum_{\substack{t\geq 1\\t^2 = 4n\\t \equiv \pm m \pmod{M}}}t^\ell\right)q^n.
\end{align*}
We have
\[
	\Sigma_1 = \sum_{b \pmod{M}}\sum_{\substack{n=1\\(n,M)=1}}^\infty\sum_{\pm}\mu_{\ell,\pm m, b}(n)q^n.
\]

Now we prove (1). Since $ m $ is odd, $ \mu_{\ell,\pm m, b}(n) = 0 $ by \Cref{lemmaLambda1} (1). The inner sum in $ \Sigma_2 $ is also zero.

Next, we prove (2). By \Cref{lemmaLambda1}, $ \mu_{\ell,\pm 2m_1, b}(n) = 0 $ unless $ b = 2b_1 $ and $ n \equiv m_1^2-b_1^2 \pmod{2^fM_1} $, hence
\begin{align*}
	\Sigma_1 = \sum_{b_1 \pmod{2^{e-1}M_1}}\sum_{\substack{n=1\\(n,M)=1\\n \equiv m_1^2-b_1^2\pmod{2^fM_1}}}^\infty\sum_{\pm}\mu_{\ell,\pm 2m_1, 2b_1}(n)q^n.
\end{align*}
Consider $ n $ such that $ n \equiv m_1^2-b_1^2 \pmod{2^fM_1} $. If $ e \geq 2 $, then $ f = e $, and $ (m_1^2-b_1^2,2^{e-1}M_1) = 1 $ if and only if $ (n,M) = 1 $. However, if $ e = 1 $, then $ f = 0 $, and $ (m_1^2-b_1^2,2^{e-1}M_1) = 1 $ if and only if $ (n,M_1) = 1 $. Thus,
\[
	\Sigma_1 = \sum_{\substack{b_1 \pmod{2^{e-1}M_1}\\(m_1^2-b_1^2,2^{e-1}M_1) = 1}}\sum_{\substack{n=1\\n \equiv 1 \pmod{2}\\n \equiv m_1^2-b_1^2\pmod{2^fM_1}}}^\infty\sum_{\pm}\mu_{\ell,\pm 2m_1, 2b_1}(n)q^n.
\]
Using \Cref{lemmaLambda1} and putting together the terms for $ b_1 $ and $ -b_1 $, we get
\begin{align*}
	\Sigma_1& = \sum_{\substack{b_1 \pmod{2^{e-1}M_1}\\(m_1^2-b_1^2,2^{e-1}M_1) = 1}}\sum_{\substack{n=1\\n \equiv 1 \pmod{2}\\n \equiv m_1^2-b_1^2\pmod{2^fM_1}}}^\infty\sum_{\pm}\left(2^\ell \sum_{\substack{d \mid n\\d < \sqrt{n}\\d \equiv \pm m_1-b_1\pmod{2^{e-1}M_1}}}d^\ell\right)q^n\\
	& = 2^\ell\sum_{\substack{b_1 \pmod{2^{e-1}M_1}\\(m_1^2-b_1^2,2^{e-1}M_1)=1}}\sum_{\substack{n=1\\n\equiv 1 \pmod{2}\\n \equiv m_1^2-b_1^2\pmod{2^fM_1}}}^\infty\sum_{\pm}\left(\sum_{\substack{d \mid n\\d < \sqrt{n}\\d \equiv \pm(m_1-b_1)\pmod{2^{e-1}M_1}}}d^\ell\right)q^n\\
	& = 2^\ell\sum_{\substack{b_1 \pmod{2^{e-1}M_1}\\(m_1^2-b_1^2,2^{e-1}M_1)=1}}G_{\ell,m_1-b_1,2^{e-1}M_1}\vert S_{2^fM_1,m_1^2-b_1^2}\vert S_{2,1}.
\end{align*}

Setting $ t = 2t_1 $ in $ \Sigma_2 $, we get
\[
	\Sigma_2 = 2^{\ell-1}\sum_{\pm}\sum_{\substack{t_1\geq 1\\t_1 \equiv \pm m_1 \pmod{2^{e-1}M_1}\\(t_1, M) = 1}}t_1^\ell q^{t_1^2} = 2^{\ell-1}T_{\ell,m_1,2^{e-1}M_1}\otimes\chi_{M,0}.\qedhere
\]
\end{proof}

We apply \Cref{lemmaLambda2} to the particular cases $ M = 6 $ and $ M = 8 $.

\begin{lemma}
\label{lemmaLambda3}
Let $ \ell \in \Z_{\geq 0} $ and $ m \in \Z $. If $ \chi_{6,0} $ denotes the principal character modulo $ 6 $, then
\[
	\Lambda_{\ell,m,6}\vert U_4\otimes \chi_{6,0} = \begin{cases}
		2^{\ell+1}G_{\ell,1,3}\vert S_{6,5}&\text{if }m \equiv 0 \pmod{6},\\
		0&\text{if }m \equiv 1,3,5 \pmod{6}\\
		2^\ell G_{\ell,1,3}\vert S_{6,1}+2^{\ell-1}T_{\ell,1,6}&\text{if }m \equiv 2,4 \pmod{6}.
	\end{cases}
\]
\end{lemma}
\begin{proof}
This follows from \Cref{lemmaLambda2} with $ e = 1 $, $ M_1 = 3 $, and $ f = 0 $.

For $ m \equiv 0 \pmod{6} $, $ b_1 $ runs over $ \{1,2\} $, hence
\[
	\Lambda_{\ell,0,6}\vert U_4\otimes \chi_{6,0} = 2^\ell\left(G_{\ell,-1,3}\vert S_{3,-1}\vert S_{2,1}+G_{\ell,-2,3}\vert S_{3,-1}\vert S_{2,1}\right) = 2^{\ell+1}G_{\ell,1,3}\vert S_{5,6}.
\]

For $ m \equiv 1,3,5 \pmod{6} $, the function is zero by \Cref{lemmaLambda2} (1).

For $ m = 2m_1 \equiv 2, 4 \pmod{6} $, only the term with $ b_1 = 0 $ appears in the sum, and we have
\[
	\Lambda_{\ell,m,6}\vert U_4\otimes \chi_{6,0} = 2^\ell G_{\ell,m_1,3}\vert S_{3,m_1^2}\vert S_{2,1}+2^{\ell-1}T_{\ell,1,3}\otimes\chi_{6,0} = 2^\ell G_{\ell,1,3}\vert S_{6,1}+2^{\ell-1}T_{\ell,1,6}.\qedhere
\]
\end{proof}

\begin{lemma}
\label{lemmaLambda4}
Let $ \ell \in \Z_{\geq 0} $ and $ m \in \Z $. If $ \chi_{8,0} $ denotes the principal character modulo $ 8 $, then
\[
	\Lambda_{\ell,m,8}\vert U_4\otimes\chi_{8,0} = \begin{cases}
		2^{\ell+1}G_{\ell,1,4}\vert S_{8,7}&\text{if }m \equiv 0 \pmod{8},\\
		0&\text{if }m \equiv 1,3,5,7 \pmod{8},\\
		2^\ell G_{\ell,1,4}\vert S_{4,1}+2^{\ell-1}T_{\ell,1,4}&\text{if }m \equiv 2, 6 \pmod{8},\\
		2^{\ell+1}G_{\ell,1,4}\vert S_{8,3}&\text{if }m \equiv 4 \pmod{8}.
	\end{cases}
\]
\end{lemma}
\begin{proof}
This follows from \Cref{lemmaLambda2} with $ e = 3 $, $ M_1 = 1 $, and $ f = 3 $.

For $ m \equiv 0 \pmod{8} $, $ b_1 $ runs over $ \{1,3\} $, hence
\[
	\Lambda_{\ell,0,8}\vert U_4\otimes\chi_{8,0} = 2^\ell\left(G_{\ell,-1,4}\vert S_{8,-1}+G_{\ell,-3,4}\vert S_{8,-1}\right) = 2^{\ell+1}G_{\ell,1,4}\vert S_{8,7}.
\]

For $ m $ odd, the function is zero by \Cref{lemmaLambda2} (1).

For $ m = 2m_1 \equiv 2,6 \pmod{8} $, $ b_1 $ runs over $ \{0, 2\} $, hence
\begin{align*}
	\Lambda_{\ell,m,8}\vert U_4\otimes\chi_{8,0}& = 2^{\ell}\left(G_{\ell,m_1,4}\vert S_{8,m_1^2}+G_{\ell,m_1-2,4}\vert S_{8,m_1^2-4}\right)+2^{\ell-1}T_{\ell,m_1,4}\\
	& = 2^\ell\left(G_{\ell,1,4}\vert S_{8,1}+G_{\ell,1,4}\vert S_{8,5}\right)+2^{\ell-1}T_{\ell,1,4}\\
	& = 2^\ell G_{\ell,1,4}\vert S_{4,1}+2^{\ell-1}T_{\ell,1,4}.
\end{align*}

For $ m = 2m_1 = 4 $, $ b_1 $ runs over $ \{1,3\} $, hence
\[
	\Lambda_{\ell,4,8}\vert U_4\otimes \chi_{8,0} = 2^\ell\left(G_{\ell,1,4}\vert S_{8,3}+G_{\ell,-1,4}\vert S_{8,-5}\right) = 2^{\ell+1}G_{\ell,1,4}\vert S_{8,3}.\qedhere
\]
\end{proof}

To prove the explicit formulas for $ H_{m,M}(p) $ in \Cref{secExp}, we twist the quasimodular form in \Cref{lemmaHLambda} by the principal character $ \chi_{M,0} $. It is necessary to know to which space the resulting function belongs. The answer is provided by the next lemma. An analogous result for weight $ \kappa \in \frac{1}{2}+\Z $ is \cite[Lemma 2.3 (2)]{BK2}.

\begin{lemma}
\label{lemmaS}
Let $ \kappa \in \Z $, $ N_1, N_2 \in \Z_{\geq 1} $, and $ \chi $ be a Dirichlet character of conductor $ N_\chi $ such that $ N_\chi \mid N_2 \mid N_1 $. Let $ M \in \Z_{\geq 1} $ and $ m \in \Z $. If $ f = \sum_{n=0}^\infty a(n)q^n $ satisfies weight $ \kappa $ modularity on $ \Gamma_0(N_1)\cap\Gamma_1(N_2) $ with character $\chi $ and
\[
	\Gamma := \Gamma_0(\lcm(N_1,M^2,MN_2,MN_\chi))\cap \Gamma_1(\lcm(N_2,M)),
\]
then $ f\vert S_{M,m} $ satisfies weight $ \kappa $ modularity on $ \Gamma $ with character $ \chi $. In particular, if $ \chi_{M,0} $ denotes the principal character modulo $ M $, then $ f\otimes\chi_{M,0} $ satisfies weight $ \kappa $ modularity on $ \Gamma $ with character $ \chi $.

Moreover, if $ M \mid 24 $, then we can replace $ \Gamma $ with $ \Gamma_0(\lcm(N_1,M^2,MN_2,MN_\chi))\cap\Gamma_1(N_2) $.
\end{lemma}
\begin{proof}
We show that if $ \gamma := \begin{pmatrix}a&b\\c&d\end{pmatrix} \in \Gamma_0(\lcm(N_1,M^2,MN_2,MN_\chi))\cap \Gamma_1(N_2) $, then
\[
	\left(f\vert S_{M,m}\right)\vert_\kappa \gamma = \chi(d)f\vert S_{M,ma^2}.
\]
We have
\begin{align*}
	\left(f\vert S_{M,m}\right)(\tau)& = \sum_{n=0}^\infty a(n)q^n\frac{1}{M}\sum_{\ell\pmod{M}}e^{\frac{-2\pi im\ell}{M}}e^{\frac{2\pi i n\ell}{M}} = \frac{1}{M}\sum_{\ell\pmod{M}}e^{\frac{-2\pi im\ell}{M}}f\left(\tau+\frac{\ell}{M}\right)\\
	& = \frac{1}{M}\sum_{\ell\pmod{M}}e^{\frac{-2\pi im\ell}{M}}f\bigg\vert_\kappa \begin{pmatrix}1&\frac{\ell}{M}\\0&1\end{pmatrix}(\tau),
\end{align*}
hence
\[
	(f\vert S_{M,m})|_\kappa \gamma = \frac{1}{M}\sum_{\ell \pmod{M}}e^{\frac{-2\pi im\ell}{M}}f\bigg\vert_\kappa \begin{pmatrix}1&\frac{\ell}{M}\\0&1\end{pmatrix}\gamma.
\]
For every $ \ell = 0, \dots, M-1 $, let $ \ell' := d^2\ell $. Since $ ad \equiv 1 \pmod{M} $, we have $ a\ell' \equiv d\ell \pmod{M} $. Let $ \gamma' $ be a $ 2 \times 2 $ matrix of determinant $ 1 $ such that
\[
	\begin{pmatrix}1&\frac{\ell}{M}\\0&1\end{pmatrix}\gamma = \gamma'\begin{pmatrix}1&\frac{\ell'}{M}\\0&1\end{pmatrix},
\]
i.e.,
\[
	\gamma' = \begin{pmatrix}
	1&\frac{\ell}{M}\\
	0&1
	\end{pmatrix}\begin{pmatrix}
	a&b\\
	c&d
	\end{pmatrix}\begin{pmatrix}
	1&-\frac{\ell'}{M}\\
	0&1
	\end{pmatrix} = \begin{pmatrix}
	a+\frac{c\ell}{M}&-\frac{\ell' a}{M}-\frac{c\ell\ell'}{M^2}+b+\frac{d\ell}{M}\\
	c&-\frac{c\ell'}{M}+d
	\end{pmatrix}.
\]
The condition $ a\ell' \equiv d\ell \pmod{M} $ together with $ M^2 \mid c $ implies that the entries of $ \gamma' $ are integers. Because $ MN_2 \mid c $, we have $ a+\frac{c\ell}{M} \equiv a \equiv 1 \pmod{N_2} $ and $ -\frac{c\ell}{M}+d \equiv d \equiv 1\pmod{N_2} $, and we get $ \gamma' \in \Gamma_0(N_1)\cap\Gamma_1(N_2) $. Thus,
\[
	(f\vert S_{M,m})\vert_\kappa \gamma = \frac{1}{M}\sum_{\ell \pmod{M}}e^{\frac{-2\pi im\ell}{M}}f\bigg\vert_\kappa \gamma'\begin{pmatrix}1&\frac{\ell'}{M}\\0&1\end{pmatrix} = \chi\left(\frac{-c\ell'}{M}+d\right)\frac{1}{M}\sum_{\ell \pmod{M}}e^{\frac{-2\pi im\ell}{M}}f\bigg\vert_\kappa \begin{pmatrix}1&\frac{\ell'}{M}\\0&1\end{pmatrix}.
\]
We assume $ MN_\chi \mid c $, hence $ \chi\left(\frac{-c\ell'}{M}+d\right) = \chi(d) $. Changing the index of summation to $ \ell' $ and using $ \ell \equiv d^{-2}\ell' \equiv a^2\ell' \pmod{M} $, we get
\[
	(f\vert S_{M,m})\vert_\kappa \gamma = \chi(d)\frac{1}{M}\sum_{\ell' \pmod{M}}e^{\frac{-2\pi i ma^2\ell}{M}}f\bigg\vert_\kappa\begin{pmatrix}1&\frac{\ell'}{M}\\0&1\end{pmatrix} = \chi(d)f\vert S_{M,ma^2}.
\]
If we additionally assume $ \gamma \in \Gamma_1(M) $, then $ a \equiv 1 \pmod{M} $, hence
\[
	(f\vert S_{M,m})\vert_\kappa\gamma = \chi(d)f\vert S_{M,m}
\]
and  $ f\vert S_{M,m} $ satisfies weight $ \kappa $ modularity on $ \Gamma $. To prove the ``moreover" part, we note that if $ M \mid 24 $, then $ a^2 \equiv 1 \pmod{M} $ for every $ a $ such that $ (a,M) = 1 $, hence $ f\vert S_{M,ma^2} = f\vert S_{M,m} $.
\end{proof}

\section{Explicit formulas modulo 6 and 8}
\label{secExp}

In this section, we prove the explicit formulas for $ H_{m,M}(p) $, where $ M \in \{6,8\} $ and $ p $ is a prime. We begin by proving that certain functions are CM cusp forms and identify the spaces to which they belong.

For $ k \in \Z_{\geq 2} $ and an odd Dirichlet character $ \chi $, let
\[
	\psi_k(\chi,n) := \sum_{x^2+ky^2 = n}\chi(x)x,
\]
where the sum runs over all $ x, y \in \Z $ such that $ x^2+ky^2 = n $, and
\[
	\Psi_k(\chi, \tau) := \frac{1}{2}\sum_{n=1}^\infty \psi_k(\chi,n)q^n.
\]
The $ \frac{1}{2} $ in front of the sum normalizes it in such a way that the coefficient of $ q $ equals $ 1 $. If we set
\[
	\theta_0(\tau) := \sum_{y \in \Z}q^{y^2}
\]
and
\[
	\theta(\chi,1,\tau) := \frac{1}{2}\sum_{x\in\Z}\chi(x)xq^{x^2},
\]
then we see that
\[
	\Psi_k(\chi,\tau) = \theta(\chi,1,\tau)\cdot\left(\theta_0\vert V_k\right)(\tau).
\]

\begin{lemma}
\label{lemmaCM}
Let $ \chi_{-3} $ and $ \chi_{-4} $ denote the non-principal characters modulo $ 3 $ and $ 4 $, respectively. The functions $ \Psi_k(\chi, \tau) $ satisfy the following:
\begin{enumerate}
\item $ \Psi_3(\chi_{-3},\tau) $ is a newform in $ S_2(\Gamma_0(36)) $ and has complex multiplication by $ \Q(\sqrt{-3}) $ (LMFDB label 36.2.a.a).
\item $ \Psi_4(\chi_{-4},\tau) $ is a newform in $ S_2(\Gamma_0(64)) $ and has complex multiplication by $ \Q(i) $ (LMFDB label 64.2.a.a).
\item $ \Psi_2(\chi_{-4},\tau) \in S_2(\Gamma_0(64),\legendre{2}{\cdot}) $. It satisfies
\[
	\Psi_2(\chi_{-4},\tau) = \frac{f(\tau)+\overline{f}(\tau)}{2}+\frac{f(\tau)-\overline{f}(\tau)}{2i},
\]
where $ f(\tau) = \sum_{n=1}^\infty a(n)q^n $ is the newform with LMFDB label 64.2.b.a and $ \overline{f}(\tau) = \sum_{n=1}^\infty \overline{a(n)}q^n $. The form $ f $ has complex multiplication by $ \Q(\sqrt{-2}) $.
\end{enumerate}
\end{lemma}
\begin{proof}
If $ \chi $ is an odd character of conductor $ N_{\chi} $, then by \cite[Theorem 1.44]{On}, $ \theta(\chi,1,\tau) \in S_{\frac{3}{2}}(\Gamma_0(4N_{\chi}^2),\chi\chi_{-4}) $. By \cite[Proposition 1.41]{On}, $ \theta_0(\tau) \in M_{\frac{1}{2}}(\Gamma_0(4)) $, hence $ \left(\theta_0\vert V_k\right)(\tau) \in M_{\frac{1}{2}}(\Gamma_0(4k),\legendre{k}{\cdot}) $.

Thus, if $ k = 3 $ and $ \chi = \chi_{-3} $, then $ \theta(\chi_{-3},1,\tau) \in S_{\frac{3}{2}}(\Gamma_0(36),\legendre{12}{\cdot}) $ and $ \left(\theta_0\vert V_3\right)(\tau) \in M_{\frac{1}{2}}(\Gamma_0(12),\legendre{3}{\cdot}) $, hence $ \Psi_3(\chi_{-3},\tau) \in S_2(\Gamma_0(36)) $.

If $ \chi = \chi_{-4} $, then $ \theta(\chi_{-4},1,\tau) \in S_{\frac{3}{2}}(\Gamma_0(64),\chi_{-4}^2) = S_{\frac{3}{2}}(\Gamma_0(64)) $. For $ k = 4 $, we get $ \left(\theta_0\vert V_4\right)(\tau) \in M_{\frac{1}{2}}(\Gamma_0(16)) $, hence $ \Psi_4(\chi_{-4},\tau) \in S_2(\Gamma_0(64)) $.

For $ k = 2 $, we get $ \left(\theta_0\vert V_2\right)(\tau) \in M_{\frac{1}{2}}(\Gamma_0(8),\legendre{2}{\cdot}) $, hence $ \Psi_2(\chi_{-4},\tau) \in S_2(\Gamma_0(64),\legendre{2}{\cdot}) $.

Next, we identify these forms with forms in the LMFDB \cite{LMFDB}. We have $ [\SL_2(\Z):\Gamma_0(36)] = 72 $ and $ [\SL_2(\Z):\Gamma_0(64)] = 96 $ by \Cref{lemmaInd}. To show that two forms on $ \Gamma_0(36) $ (respectively $ \Gamma_0(64) $) are equal by \Cref{thmSturm}, we must check the first $ \frac{72}{6} = 12 $ (respectively $ \frac{96}{6} = 16 $) coefficients. Since
\begin{align*}
	\Psi_3(\chi_{-3},\tau)& = q-4q^7+O\left(q^{12}\right),\\
	\Psi_4(\chi_{-4},\tau)& = q+2q^5-3q^9-6q^{13}+O\left(q^{16}\right),\\
	\Psi_2(\chi_{-4},\tau)& = q+2q^3-q^9-6q^{11}+O\left(q^{16}\right),
\end{align*}
$ \Psi_3(\chi_{-3},\tau) $ is equal to the form with LMFDB label 36.2.a.a. and $ \Psi_4(\chi_{-4},\tau) $ to the form with LMFDB label 64.2.a.a. These are newforms with complex multiplication by $ \Q(\sqrt{-3}) $ and $ \Q(i) $, respectively. Finally, the form $ f $ with LMFDB label 64.2.b.a has $ q $-expansion
\[
	f(\tau) = q+2iq^3-q^9-6iq^{11}+O\left(q^{16}\right),
\]
hence
\[
	\overline{f}(\tau) = q-2iq^3-q^9+6iq^{11}+O\left(q^{16}\right).
\]
The formula for $ \Psi_2(\chi_{-4},\tau) $ follows.
\end{proof}

The function
\[
	D(\tau) := \sum_{n = 1}^\infty \sigma(n)q^n,
\]
where $ \sigma(n) = \sum_{d \mid n}d $, is an example of a quasimodular form. Let
\[
	E_2(\tau) := 1-24\sum_{n=1}^\infty \sigma(n)q^n
\]
be the Eisenstein series of weight $ 2 $. If $ \tau = u+iv $, then
\[
	\widehat{E_2}(\tau) := E_2(\tau)-\frac{3}{\pi v}
\]
satisfies weight $ 2 $ modularity on $ \SL_2(\Z) $ \cite[p. 113]{Ko}. We have $ D(\tau) = \frac{1}{24}-\frac{1}{24}E_2(\tau) $, hence
\[
	\widehat{D}(\tau) := D(\tau)-\frac{1}{24}+\frac{1}{8\pi v} = -\frac{1}{24}\widehat{E_2}(\tau)
\]
also satisfies weight $ 2 $ modularity on $ \SL_2(\Z) $.

\begin{lemma}
\label{lemmaD}
If $ M \in \Z_{\geq 2} $ and $ m \in \Z $ is such that $ M \nmid m $, then $ D\vert S_{M,m} \in M_2(\Gamma_0(M^2)) $. In particular, if $ \chi_{M,0} $ denotes the principal character modulo $ M $, then $ D \otimes \chi_{M,0} \in M_2(\Gamma_0(M^2)) $.
\end{lemma}
\begin{proof}
We have
\[
	E_2(\tau)-ME_2(M\tau) = \widehat{E_2}(\tau)-M\widehat{E_2}(M\tau).
\]
The left-hand side is holomorphic and the ring-hand side satisfies weight $ 2 $ modularity on $ \Gamma_0(M) $, hence $ E_2-ME_2\vert V_M $ is a holomorphic modular form and belongs to $ M_2(\Gamma_0(M)) $. Since $ M \nmid m $, we get
\[
	E_2\vert S_{M,m} = \left(E_2-ME_2\vert V_M\right)\vert S_{M,m} \in M_2(\Gamma_0(M^2))
\]
by \Cref{lemmaS}, hence also $ D\vert S_{M,m} = -\frac{1}{24}E_2 \vert S_{M,m} \in M_2(\Gamma_0(M^2)) $.
\end{proof}

It is enough to determine $ H_{m,M}(n) $ for $ 0 \leq m \leq \frac{M}{2} $ since, by definition,
\[
	H_{m,M}(n) = \sum_{\substack{t \in \Z\\t\equiv m \pmod{M}}}H(4n-t^2) = H_{-m,M}(n).
\]

\begin{proposition}
\label{thmMod6}
If $ \chi_{6,0} $ denotes the principal character modulo $ 6 $, then
\begin{align*}
	\left(\mathcal{H}\theta_{0,6}\right)\vert U_4\otimes \chi_{6,0}+2G_{1,1,3}\vert S_{6,5}& = \frac{1}{3}D\vert S_{6,1}+\frac{2}{3}D\vert S_{6,5}+\frac{1}{6}\Psi_3(\chi_{-3},\cdot),\\
	\left(\H\theta_{1,6}\right)\vert U_4\otimes\chi_{6,0}& = \frac{1}{4}D\vert S_{6,1}+\frac{1}{6}D\vert S_{6,5}+\frac{1}{12}\Psi_3(\chi_{-3},\cdot),\\
	\left(\H\theta_{2,6}\right)\vert U_4\otimes\chi_{6,0}+G_{1,1,3}\vert S_{6,1}+\frac{1}{2}T_{1,1,6}& = \frac{1}{2}D\vert S_{6,1}+\frac{1}{3}D\vert S_{6,5}-\frac{1}{12}\Psi_3(\chi_{-3},\cdot),\\
	\left(\H\theta_{3,6}\right)\vert U_4\otimes\chi_{6,0}& = \frac{1}{6}D\vert S_{6,1}+\frac{1}{3}D\vert S_{6,5}-\frac{1}{6}\Psi_3(\chi_{-3},\cdot).
\end{align*}
\end{proposition}
\begin{proof}
We prove in detail the first formula. By \Cref{lemmaHLambda}, the function
\[
	\left(\H\theta_{0,6}+\frac{1}{2}\Lambda_{1,0,6}\right)\bigg\vert U_4
\]
is a quasimodular form of weight $ 2 $ on $ \Gamma_0(4\cdot 6^2) $. Since $ 6 $ divides $ 24 $,
\[
	\left(\H\theta_{0,6}+\frac{1}{2}\Lambda_{1,0,6}\right)\bigg\vert U_4\otimes\chi_{6,0}
\]
also satisfies weight $ 2 $ modularity on $ \Gamma_0(4\cdot 6^2) $ by the ``moreover" part of \Cref{lemmaS}. By \Cref{lemmaLambda3},
\[
	\Lambda_{1,0,6}\vert U_4 \otimes \chi_{6,0} = 4G_{1,1,3}\vert S_{6,5},
\]
hence
\[
	(\H\theta_{0,6})\vert U_4\otimes\chi_{6,0}+2G_{1,1,3}\vert S_{6,5} \in M_2(\Gamma_0(4 \cdot 36)).
\]
By \Cref{lemmaD}, $ D\vert S_{6,r} \in M_2(\Gamma_0(36)) $ for $ r \in \{1,5\} $, and by \Cref{lemmaCM} (1), $ \Psi_3(\chi_{-3},\tau) \in S_2(\Gamma_0(36)) $. Thus, both the left-hand side and the right-hand side are modular forms in $ M_2(\Gamma_0(4\cdot 36)) $.

If $ n \equiv 1 \pmod{6} $, then the coefficient of $ q^n $ on the left-hand side equals $ H_{0,6}(n) $, while the coefficient of $ q^n $ on the right-hand side equals $ \frac{1}{3}\sigma(n)+\frac{1}{12}\psi_3(\chi_{-3},n) $.

If $ n \equiv 5 \pmod{6} $, then the coefficient of $ q^n $ on the left-hand side equals $ H_{0,6}(n)+2g_{1,0,6}(n) $, while the coefficient of $ q^n $ on the right-hand side equals $ \frac{2}{3}\sigma(n)+\frac{1}{12}\psi_3(\chi_{-3},n) $.

We have $ [\SL_2(\Z):\Gamma_0(4\cdot 36)] = 288 $ by \Cref{lemmaInd}, hence we have to compare the first $ \frac{288}{6} = 48 $ coefficients to conclude that the two modular forms are equal by \Cref{thmSturm}. The proof of the remaining formulas is similar.
\end{proof}

We checked that the coefficients of the forms on both sides are equal in SageMath \cite{Sa}. The code is available at \url{https://github.com/zinmik/Hurwitz-class-numbers}.

Now we prove the explicit formulas for $ H_{m,6}(p) $.

\begin{theorem}
\label{thmH6}
Let $ p \geq 5 $ be a prime. For $ p \equiv 1 \pmod{3} $, let $ p = x^2+3y^2 $, where $ x, y \in \Z $. If $ \chi_{-3} $ denotes the non-principal character modulo $ 3 $, then
\[
	H_{m,6}(p) = 
	\setlength{\arraycolsep}{0pt}
	\renewcommand{\arraystretch}{1.2}
		\left\{\begin{array}{l @{\quad} l @{\quad} l}
		\frac{p+1}{3}+\frac{1}{3}\chi_{-3}(x)x&\text{if }m \equiv 0 \pmod{6},&p \equiv 1 \pmod{3},\\
		\frac{2p-4}{3}&\text{if }m \equiv 0 \pmod{6},&p \equiv 2 \pmod{3},\\
		\frac{p+1}{4}+\frac{1}{6}\chi_{-3}(x)x&\text{if }m \equiv 1, 5 \pmod{6},&p \equiv 1 \pmod{3},\\
		\frac{p+1}{6}&\text{if }m \equiv 1, 5 \pmod{6},&p \equiv 2 \pmod{3},\\
		\frac{p-1}{2}-\frac{1}{6}\chi_{-3}(x)x&\text{if }m \equiv 2, 4 \pmod{6},&p \equiv 1 \pmod{3},\\
		\frac{p+1}{3}&\text{if }m \equiv 2,3,4 \pmod{6},&p \equiv 2 \pmod{3},\\
		\frac{p+1}{6}-\frac{1}{3}\chi_{-3}(x)x&\text{if }m \equiv 3 \pmod{6},&p \equiv 1 \pmod{3}.
	\end{array}\right.
\]
\end{theorem}
\begin{proof}
We prove the formula in the case $ m \equiv 0 \pmod{6} $. It follows from \Cref{thmMod6} that if $ p \equiv 1 \pmod{6} $, then
\[
	H_{0,6}(p) = \frac{1}{3}\sigma(p)+\frac{1}{12}\psi_3(\chi_{-3},p) = \frac{p+1}{3}+\frac{1}{3}\chi_{-3}(x)x
\]
and if $ p \equiv 5 \pmod{6} $, then
\[
	H_{0,6}(p)+2 = \frac{2}{3}\sigma(p) = \frac{2p+2}{3}.
\]
The other cases follow from the other formulas in \Cref{thmMod6}.
\end{proof}

\begin{proposition}
\label{thmMod8}
If $ \chi_{8,0} $ denotes the principal character modulo $ 8 $, then
\begin{align*}
	(\mathcal{H}\theta_{0,8})\vert U_4\otimes\chi_{8,0}+2G_{1,1,4}\vert S_{8,7}& = \frac{1}{4}D\vert S_{4,1}+\frac{1}{3}D\vert S_{8,3}+\frac{1}{2}D\vert S_{8,7}+\frac{1}{4}\Psi_4(\chi_{-4},\cdot),\\
	\left(\H\theta_{1,8}\right)\vert U_4 \otimes \chi_{8,0}& = \frac{1}{6}D\otimes\chi_{8,0}+\frac{1}{6}\Psi_2(\chi_{-4},\cdot),\\
	\left(\H\theta_{2,8}\right)\vert U_4 \otimes \chi_{8,0}+G_{1,1,4}\vert S_{4,1}+\frac{1}{2}T_{1,1,4}& = \frac{5}{12}D\vert S_{4,1}+\frac{1}{4}D\vert S_{4,3},\\
	\left(\H\theta_{3,8}\right)\vert U_4 \otimes \chi_{8,0}& = \frac{1}{6}D\otimes\chi_{8,0}-\frac{1}{6}\Psi_2(\chi_{-4},\cdot),\\
	\left(\H\theta_{4,8}\right)\vert U_4 \otimes \chi_{8,0}+2G_{1,1,4}\vert S_{8,3}& = \frac{1}{4}D\vert S_{4,1}+\frac{1}{2}D\vert S_{8,3}+\frac{1}{3}D \vert S_{8,7}-\frac{1}{4}\Psi_4(\chi_{-4},\cdot).
\end{align*}
\end{proposition}
\begin{proof}
Let us prove the first formula. By \Cref{lemmaHLambda}, the function
\[
	\left(\H\theta_{0,8}+\frac{1}{2}\Lambda_{1,0,8}\right)\bigg\vert U_4
\]
is a quasimodular form of weight $ 2 $ on $ \Gamma_0(4\cdot 8^2) $. Since $ 8 $ divides $ 24 $,
\[
	\left(\H\theta_{0,8}+\frac{1}{2}\Lambda_{1,0,8}\right)\bigg\vert U_4\otimes\chi_{8,0}
\]
satisfies weight $ 2 $ modularity on $ \Gamma_0(4\cdot 64) $ by the ``moreover" part of \Cref{lemmaS}. By \Cref{lemmaLambda4},
\[
	\Lambda_{1,0,8}\vert U_4\otimes\chi_{8,0} = 4G_{1,1,4}\vert S_{8,7},
\]
hence
\[
	(\H\theta_{0,8})\vert U_4\otimes\chi_{8,0}+2G_{1,1,4}\vert S_{8,7} \in M_2(\Gamma_0(4\cdot 64)).
\]
By \Cref{lemmaD}, $ D\vert S_{4,1} \in M_2(\Gamma_0(16)) $ and $ D\vert S_{8,r} \in M_2(\Gamma_0(64)) $ for $ r \in \{3,7\} $. By \Cref{lemmaCM} (2), $ \Psi_4(\chi_{-4},\tau) \in S_2(\Gamma_0(64)) $. Thus, both the left-hand side and the right-hand side are modular forms in $ M_2(\Gamma_0(4\cdot 64)) $.

The coefficient of $ q^n $ on the left-hand side equals
\[
	H_{0,8}(n)+\begin{cases}
		0&\text{if }n \equiv 1,3,5 \pmod{8},\\
		2g_{1,1,4}(n)&\text{if }n \equiv 7 \pmod{8},
	\end{cases}
\]
while the coefficient of $ q^n $ on the right-hand side equals
\[
	\frac{1}{8}\psi_4(\chi_{-4},n)+\begin{cases}
		\frac{1}{4}\sigma(n)&\text{if }n \equiv 1 \pmod{4},\\
		\frac{1}{3}\sigma(n)&\text{if }n \equiv 3 \pmod{8},\\
		\frac{1}{2}\sigma(n)&\text{if }n \equiv 7 \pmod{8}.
	\end{cases}
\]
We have $ [\SL_2(\Z):\Gamma_0(4\cdot 64)] = 384 $ by \Cref{lemmaInd}, hence we have to check the first $ \frac{384}{6} = 64 $ coefficients to conclude that the two modular forms are equal by \Cref{thmSturm}.

We also prove the second formula. By \Cref{lemmaHLambda}, the function
\[
	\left(\H\theta_{1,8}+\frac{1}{2}\Lambda_{1,1,8}\right)\bigg\vert U_4
\]
is a quasimodular form of weight $ 2 $ on $ \Gamma_0(4\cdot 8^2)\cap \Gamma_1(8) $. By \Cref{lemmaLambda4},
\[
	\Lambda_{1,1,8}\vert U_4\otimes\chi_{8,0} = 0,
\]
hence if we let $ \Gamma(8) := \Gamma_0(4\cdot 64)\cap\Gamma_1(8) $, then by \Cref{lemmaS},
\[
	(\H\theta_{1,8})\vert U_4\otimes\chi_{8,0} \in M_2(\Gamma(8)).
\]
By \Cref{lemmaD}, $ D\otimes\chi_{8,0} \in M_2(\Gamma_0(64)) \subset M_2(\Gamma(8)) $. By \Cref{lemmaCM} (3), $ \Psi_2(\chi_{-4},\tau) \in S_2(\Gamma_0(64),\legendre{2}{\cdot}) $. We have $ M_2(\Gamma_0(64),\legendre{2}{\cdot}) \subset M_2(\Gamma(8)) $ because if $ \gamma = \begin{pmatrix}a&b\\c&d\end{pmatrix} \in \Gamma(8) $, then $ d \equiv 1 \pmod{8} $, hence $ \legendre{2}{d} = 1 $. Thus, both sides belong to $ M_2(\Gamma(8)) $ and the rest of the proof is similar.
\end{proof}

\begin{theorem}
\label{thmH81}
Let $ p \geq 3 $ be a prime. For $ p \equiv 1 \pmod{4} $, let $ p = x^2+4y^2 $, where $ x, y \in \Z $. If $ \chi_{-4} $ denotes the non-principal character modulo $ 4 $, then
\[
	H_{m,8}(p) =
	\setlength{\arraycolsep}{0pt}
	\renewcommand{\arraystretch}{1.2}
	\left\{\begin{array}{l @{\quad} l @{\quad} l}
		\frac{p+1}{4}+\frac{1}{2}\chi_{-4}(x)x&\text{if }m \equiv 0 \pmod{8},&p \equiv 1 \pmod{4},\\
		\frac{p+1}{3}&\text{if }m \equiv 0 \pmod{8},&p \equiv 3 \pmod{8},\\
		\frac{p-3}{2}&\text{if }m \equiv 0 \pmod{8},&p \equiv 7 \pmod{8},\\
		\frac{5p-7}{12}&\text{if }m \equiv 2, 6 \pmod{8},&p \equiv 1 \pmod{4},\\
		\frac{p+1}{4}&\text{if }m \equiv 2,6 \pmod{8},&p \equiv 3 \pmod{4},\\
		\frac{p+1}{4}-\frac{1}{2}\chi_{-4}(x)x&\text{if }m \equiv 4 \pmod{8},&p \equiv 1 \pmod{4},\\
		\frac{p-3}{2}&\text{if }m \equiv 4 \pmod{8},&p \equiv 3 \pmod{8},\\
		\frac{p+1}{3}&\text{if }m \equiv 4 \pmod{8},&p \equiv 7 \pmod{8}.
	\end{array}\right.
\]
\end{theorem}
\begin{proof}
We prove the formula in the case $ m \equiv 0 \pmod{8} $. It follows from \Cref{thmMod8} that if $ p \equiv 1 \pmod{4} $, then
\[
	H_{0,8}(p) = \frac{1}{4}\sigma(p)+\frac{1}{8}\psi_4(\chi_{-4},p) = \frac{p+1}{4}+\frac{1}{2}\chi_{-4}(x)x,
\]
if $ p \equiv 3 \pmod{8} $, then
\[
	H_{0,8}(p) = \frac{1}{3}\sigma(p) = \frac{p+1}{3},
\]
and if $ p \equiv 7 \pmod{8} $, then
\[
	H_{0,8}(p)+2 = \frac{1}{2}\sigma(p) = \frac{p+1}{2}.\qedhere
\]
\end{proof}

\begin{theorem}
\label{thmH82}
Let $ p \geq 3 $ be a prime. For $ p \equiv 1, 3 \pmod{8} $, let $ p = x^2+2y^2 $, where $ x, y \in \Z $. If $ \chi_{-4} $ denotes the non-principal character modulo $ 4 $, then
\[
	H_{m,8}(p) =
	\setlength{\arraycolsep}{0pt}
	\renewcommand{\arraystretch}{1.2}
	\left\{\begin{array}{l @{\quad} l @{\quad} l}
		\frac{p+1}{6}+\frac{1}{3}\chi_{-4}(x)x&\text{if }m \equiv 1,7 \pmod{8},&p \equiv 1,3 \pmod{8},\\
		\frac{p+1}{6}&\text{if }m \equiv 1,3,5,7 \pmod{8},&p \equiv 5,7 \pmod{8},\\
		\frac{p+1}{6}-\frac{1}{3}\chi_{-4}(x)x&\text{if }m \equiv 3,5 \pmod{8},&p \equiv 1,3 \pmod{8}.
	\end{array}\right.
\]
\end{theorem}
\begin{proof}
We prove the formula in the case $ m \equiv 1 \pmod{8} $. By \Cref{thmMod8},
\[
	H_{1,8}(p) = \frac{1}{6}\sigma(p)+\frac{1}{12}\psi_2(\chi_{-4},p) = \frac{p+1}{6}+\begin{cases}
		\frac{1}{3}\chi_{-4}(x)x&\text{if }p\equiv 1,3 \pmod{8},\\
		0&\text{if }p \equiv 5, 7 \pmod{8}.
	\end{cases}\qedhere
\]
\end{proof}

\end{document}